\documentclass[11pt]{article}
\usepackage{amsmath}
\usepackage{amscd}
\usepackage{amsfonts}
\usepackage{amsmath,amsthm,hyperref}
\usepackage[all]{xy}

\newtheorem{theorem}{Theorem}[section]
\newtheorem{proposition}[theorem]{Proposition}

\newtheorem{corollary}[theorem]{Corollary}

\theoremstyle{remark}
\newtheorem{remark}[theorem]{Remark}
\newtheorem{example}[theorem]{\bf Example}

\begin{document}

\title{\bf{On the Willmore functional of 2-tori in some product Riemannian manifolds}}
\author{{Peng Wang\footnote{ Supported by the Tianyuan Foundation of China, grant no. 10926112 and the Program for Young Excellent Talents in Tongji University.} }}
\date{}
\maketitle
\begin{center}
{\bf Abstract}
\end{center}

We discuss the minimum of Willmore functional of torus in a Riemannian
manifold $N$, especially for the case that $N$ is a product
manifold. We show that when $N=S^2\times S^1$, the minimum of
$W(T^2)$ is 0, and when $N=R^2\times S^1$, there exists no torus
having least Willmore functional. When $N=H^2(-c)\times S^1$, and $x=\gamma\times S^1$, the minimum of $W(x)$ is
$2\pi^2\sqrt{c}$.\\


{\bf MSC(2000):\hspace{2mm} 53A30, 53B30}

\section{Introduction}

There are many interesting and important results and conjectures on the global properties of surfaces in a Riemannian manifold.
The classical Willmore conjecture is one of famous open
problems in this direction, concerning the minimum of Willmore functional
of a torus in $S^3$. Willmore functional is known as an important
conformal invariant in the conformal geometry of submanifolds and there have been
many works on the global geometry of Willmore surfaces in space
forms, see \cite{Kusner1}, \cite{Li-Yau}, \cite{Wang1998}, etc. There have been many classical results as to the Willmore
conjecture, i.e., the minimum of Willmore functional of a torus in
$S^n$ and its relation with the conformal structure of the torus,
see \cite{Kusner1}, \cite{Li-Yau}, etc.

Till now there are few results on the global properties of Willmore surfaces in
generic Riemannian manifolds. And the influence of the conformal structure
of $S^n$ or the Riemannian manifold on the infimum is also not
clear. As to this problem, some relevant works can be found in
\cite{Barros}, \cite{Tai}, \cite{Hu-Li}, \cite{Montiel}. We also note that there are several works on the Willmore functionals via PDE, see \cite{Mondino1,Mondino2}.

In \cite{Barros}, similar to Pinkall's idea \cite{Pin}, Barros
considered the behavior of Willmore functional of torus in $S^3$ with
different conformal metrics. By use of Hopf map, he considered a kind of
metrics $g_t$ on $S^3$ which are not conformal to each other. Note that
when $t=1$, $g_1$ is just the standard metric on $S^3$, and when $t\neq 1$, $(S^3,g_t)$ is known as the Berger sphere.  He showed
that under such metrics, the minimum of Willmore functional of tori
derived by closed curves via the Hopf map is $2\pi^2 t^2$. So
different from the Willmore conjecture, when $t<1$, the minimum is
less than $2\pi^2$. His result gives an interesting description on the
behavior of Willmore functional of torus in $S^n$ with different
conformal metrics. There are few further research. Here we
will discuss furthermore the behavior of Willmore functional of a
torus in a Riemannian manifold.

As to this problem, we find that there are two different cases.
In the first case, there exist tori attaining the minimum of Willmore
functional in the Riemannian manifold, for example, the Riemanniani
space forms, and some kind of product manifolds as below. In the
second case, there exist Riemannian manifolds without tori in them
with least Willmore functional.

To be concrete, on the one hand, we show that, for an immersed torus in $S^2\times S^1$,
$W(T^2)\geq0$ and $W(T^2)=0$ if and only if it is congruent to the
Clifford torus $x(u,v)=(\cos u,\sin u,0,\cos v,\sin v)\subset
S^2\times S^1.$ Furthermore, for a product Riemannian manifold $M\times S^1$, if $M$ has
closed geodesics, there exist tori with $W(T^2)=0$. For
example, when $M=T^n$ or $S^n\times T^m$, etc, there are totally geodesic (and hence totally umbilic) $T^2$ in $M\times S^1$ with $W(T^2)=0$. On the other hand, we prove that all tori
in $R^2\times S^1$ satisfy $W(T^2)>0$. And there exist a kind of
tori $x_t$, $\lim_{t\rightarrow\infty} W(x_t)=0$, showing that there
exists no torus with least Willmore functional in $R^2\times S^1$.
We also show that if we change the metric on $R^2$, we can get a
metric such that $R^2\times S^1$ has totally geodesic torus.

For the case $M=H^2(-c)$, we show that for torus $x=\gamma\times
S^1\subset H^2(-c)\times S^1$, $W(x)\geq2\pi\sqrt{c}$, where
$\gamma$ is an immersed closed curve of $H^{2}(-c)$. And
$W(x)=2\pi\sqrt{c}$ if and only if $\gamma$ is congruent with a geodesic
circle of radius $\frac{\sinh^{-1}1}{\sqrt{-c}}$. In this case
the torus is a Willmore torus in $H^2(-c)\times S^1$.

 This paper is organized as follows. In Section~2, we review the general theory of
submanifolds in a Riemannian manifold and the notion of Willmore
functional. Then we discuss an estimate of Willmore
functional of Willlmore surfaces in a Riemannian manifold in Section~3. In Section~4, we
consider the minimum of Willmore functional of tori in $M\times
S^1$, mainly for the case $M=S^2,\ R^2$ and $H^2(-c)$.

\section{Variational equation for Willmore submanifolds and a first application}

In this section, using moving frames, we review the structure
equations of submanifolds in a Riemannian manifold first. Then we
give the description of Willmore functional and the variational equation of Willmore functional. For more details, we
refer to \cite{Chern}, \cite{Hu-Li}. Then as a special case, we see that totally umbilic  submanifolds give the simplest examples of Willlmore submanifolds.

Let $x:M\rightarrow N$ be an immersed submanifold  into a
Riemannian manifold $(N,g)$ with $dim M=m,\ dim N=n$.  Let $\{\theta_A\}$ be a local
orthonormal frame of $N$ such that
$$x^*\theta_{\alpha}=0,\ m+1\leq\alpha\leq n.$$
We assume the region of the index as below:
$$1\leq i,j,k,\cdots,\leq m,\ m+1\leq\alpha,\beta,\gamma,\cdots,\leq n,\ 1\leq A,B,C,\cdots,\leq n.$$
For simplicity we still denote $x^*\theta_A=\theta_A$ and etc. So we
have the first fundamental form $I$ of $x$ as $$I=x^*g=\sum_{i}(\theta_i)^2.$$ Suppose that the structure
equations and integrability equations of $N$ are as follows:

\begin{equation} \label{eq-N}
 \left\{\begin{aligned}
        d\theta_{A} &= \sum_{B}\theta_{AB}\wedge \theta_{B},\ \theta_{AB}+\theta_{BA}=0, \\
       d \theta_{AB} &= \sum_{C} \theta_{AC}\wedge  \theta_{CB}-\frac{1}{2}\sum_{CD}\tilde{R}_{ABCD} \theta_{C}\wedge \theta_{D},\tilde{R}_{ABCD}+\tilde{R}_{ABDC}=0.\\
                          \end{aligned}\right .
                          \end{equation}
Here $\theta_{AB}$ and $\tilde{R}_{ABCD}$ denote respectively the
connection and Riemannian curvature. Restricting to $M$, we have
\begin{equation}d\theta_{i}  = \sum_{j}\theta_{ij}\wedge \theta_{j},\ \theta_{ij}+\theta_{ji}=0,
 \end{equation}
 and
\begin{equation}\sum_{i} \theta_{i\alpha}\wedge  \theta_{i}=0.
 \end{equation}
So $\{\theta_{ij}\}$ is the connection of $x$. And together with
Cartan lamma, we have
\begin{equation}\theta_{i\alpha}= \sum_j h^{\alpha}_{ij}\theta_j, \
h^{\alpha}_{ij}=h^{\alpha}_{ji}.
 \end{equation}
Let $\{e_1,e_{2},\cdots,e_n\}$ be the dual orthonormal basis of
$\{\theta_{A}\}$ in $N$, i.e. $\theta_{A}(e_B)=\delta_{AB}$. Then we
can write down the second fundamental form $II$ and the mean curvature
vector $\vec{H}$ of $x$ as
\begin{equation}
II=\sum_{\alpha}\sum_{ij}h^{\alpha}_{ij}\theta_i \otimes\theta_j
e_{\alpha}, \ \
\vec{H}=\frac{1}{m}\sum_{\alpha}(\sum_{i}h^{\alpha}_{ii})e_{\alpha}=\sum_{\alpha}H^{\alpha}e_{\alpha}.
\end{equation}
Recall that  a point $p\in M$ is called an {\em umbilic} point, if $II(p)=\vec{H}(p)I(p).$  $M$ is called a  {\em totally umbilic}
submanifold if the second fundamental form
\begin{equation} II=\vec{H}I
\end{equation}
on $M$. And furthermore $M$ is called a {\em totally geodesic}
submanifold if $II\equiv0$, which is the simplest case of totally umbilic submanifolds.

We define the Riemannian curvature $R_{ijkl}$, the normal
curvature  $R_{\alpha\beta kl}$, and the covariant derivative
$h^{\alpha}_{ij,k}$ by

\begin{equation}   d \theta_{ij} - \sum_{k} \theta_{ik}\wedge
\theta_{kj}=-\frac{1}{2}\sum_{kl} R_{ijkl} \theta_{k}\wedge
\theta_{l},\ R_{ijkl}+R_{ijlk}=0, \end{equation}
\begin{equation}   d \theta_{\alpha\beta} - \sum_{\gamma} \theta_{\alpha\gamma}\wedge
\theta_{\gamma\beta}=-\frac{1}{2}\sum_{kl}R_{\alpha\beta kl}
\theta_{k}\wedge \theta_{l},\ R_{\alpha\beta kl}+ R_{\alpha\beta
lk}=0.
\end{equation}
\begin{equation}\sum_k h^{\alpha}_{ij,k}\theta_k
=dh^{\alpha}_{ij}+\sum_kh^{\alpha}_{kj}\theta_{ki}+
\sum_kh^{\alpha}_{ki}\theta_{kj}
+\sum_{\beta}h^{\beta}_{ij}\theta_{\beta\alpha }.
\end{equation}
By \eqref{eq-N} we derive the integrability equations of $x$ as
\begin{equation} \label{eq:G-C-R}
\left\{ \begin{aligned}
 \hbox{Gauss equations:  }\ &
 R_{ijkl}=\tilde{R}_{ijkl}+\sum_{\alpha}(h^{\alpha}_{ik}h^{\alpha}_{jl}-h^{\alpha}_{il}h^{\alpha}_{jk}),\\
 \hbox{Codazzi equations: }\ &
h^{\alpha}_{ij,k}-h^{\alpha}_{ik,j}=\tilde{R}_{\alpha ikj},\\
\hbox{Ricci equations: }\ &
R_{\alpha\beta ij}=\tilde{R}_{\alpha\beta ij}+\sum_{k}(h^{\alpha}_{ik}h^{\beta}_{kj}-h^{\alpha}_{jk}h^{\beta}_{ki}).\\
 \end{aligned} \right.
                          \end{equation}
Let
\begin{equation}S=|II|^2=\sum_{\alpha i
j}(h^{\alpha}_{ij})^2\end{equation} denote the norm square of $II$. Set
\begin{equation}\rho^2=S-m|\vec{H}|^2=|II-\vec{H}I|^2=
\sum_{1\leq i<j\leq m}
\left(\frac{1}{m}(h^{\alpha}_{ii}-h^{\alpha}_{jj})^2+2(h^{\alpha}_{ij})^2\right).\end{equation}
It is direct to see that $\rho\equiv0$ if and only if $x$ is totally umbilic.

From the Gauss equations in \eqref{eq:G-C-R},  we have
\begin{equation}R=\sum_{ij}\tilde{R}_{ijij}+m^2|\vec{H}|^2-S,\end{equation}
where $R$ is the scalar curvature of $M$. So
\begin{equation}\rho^2=S-m|\vec{H}|^2=m(m-1)|\vec{H}|^2-R+\sum_{ij}\tilde{R}_{ijij}.\end{equation}

 We define the {\em Willmore functional } as
\begin{equation} \label{eq:Willmore}
\begin{split}
W(x):&=\int_M (|H|^2-\frac{R}{m(m-1)}+\frac{1}{m(m-1)}\sum_{ij}\tilde{R}_{ijij})^{\frac{m}{2}} dM  \\
 &=\left(\frac{1}{m(m-1)}\right)^{\frac{m}{2}}\int_M\rho^m
dM=\left(\frac{1}{m(m-1)}\right)^{\frac{m}{2}}\int_M(S-m|\vec{H}|^2)^{\frac{m}{2}}dM.
 \end{split}
 \end{equation}
It is well known (of \cite{Chen0}, \cite{Chen1}, \cite{Hu-Li},
\cite{Wang1998} ) that $W(x)$ is invariant under conformal
transforms of $N$.  We call a submanifold $x$ {\em a Willmore
submanifold} if it is a critical point of the Willmore functional
$W(x)$.

To see the variational equation, we recall the definition of Laplacian of $\vec{H}$. Let
$$\sum_iH^{\alpha}_{,i}\theta_i:=dH^{\alpha}+\sum_{\beta}H^{\beta}\theta_{\beta\alpha},\ \sum_jH^{\alpha}_{,ij}\theta_j:=dH^{\alpha}_{,i}+\sum_{k}H^{\alpha}_{,k}\theta_{ki}+\sum_{\beta}H^{\beta}_{,i}\theta_{\beta\alpha}.$$
Then we define the Laplacian of $\vec{H}$ as
\begin{equation}\Delta^{\perp}H^{\alpha}:=\sum_{i}H^{\alpha}_{,ii}.\end{equation}
Similar, for a smooth function $f$ on $M$, the Laplacian is defined as
$$\Delta f:=\sum_i f_{i,i},\ \hbox{with}\ \sum_if\theta_i:=df,\ \sum_jf_{i,j}\theta_j:=df_i+\sum_{k}f_k\theta_{ki}.$$

In \cite{Hu-Li}, Hu and Li gave the variational equation for
Willmore submanifolds:

\begin{theorem}{(\em Hu-Li,\ \cite{Hu-Li})}
 Let $x:M\rightarrow N$ be an immersed submanifold as above. It is a Willmore submanifold if and only if
\begin{equation} \label{willmore}
\begin{split}
0=&\rho^{m-2}\left[
\underset{i,j,k,\beta}{\sum}h^{\beta}_{ij}h^{\beta}_{ik}h^{\alpha}_{kj}
+ \underset{i,j, \beta}{\sum}\tilde{R}_{\beta i\alpha
j}h^{\beta}_{ij}- \underset{i,j, \beta}{\sum}H^{\beta}
h^{\beta}_{ij}h^{\alpha}_{ij}  -\underset{i,
\beta}{\sum}H^{\beta}\tilde{R}_{\beta
i\alpha \i}-\rho^2H^{\alpha}\right]      \\
& +\underset{i,j}{\sum}\left\{
2\left(\rho^{m-2}\right)_ih^{\alpha}_{ijj}+
\left(\rho^{m-2}\right)_{i,j}h^{\alpha}_{ij}+\rho^{n-2}h^{\alpha}_{ijij}\right\}
-H^{\alpha}\Delta(\rho^{m-2})\\
&-\rho^{m-2}\Delta^{\perp}H^{\alpha}-2\underset{i}{\sum}\left(\rho^{m-2}\right)_{i}H^{\alpha}_{i},\
\ \ \hbox{for}\ \ \ m+1\leq\alpha\leq n.
 \\
 \end{split}
 \end{equation}
\end{theorem}

\begin{corollary}{(\em Hu-Li,\ \cite{Hu-Li})}
 Let $x:M^2\rightarrow N $ be an immersed surface. It is Willmore if and only if
\begin{equation} \label{willmore-s}
\begin{split}
0&=\underset{i,j}{\sum}h^{\alpha}_{ij,ij}-\Delta^{\perp}H^{\alpha}-2H^2H^{\alpha}+\underset{i,j, \beta}{\sum}\tilde{R}_{\beta i\alpha
j}h^{\beta}_{ij}-\underset{i,
\beta}{\sum}H^{\beta}\tilde{R}_{\beta
i\alpha i}\\
&+ \underset{i,j,k}{\sum}\tilde{R}_{ikjk}h^{\alpha}_{ij}+\sum_{i,j,k}H^{\beta}h^{\beta}_{ij}h^{\alpha}_{ij}-\underset{i,j}{\sum}\tilde{R}_{ijij}H^{\alpha},
\ \ \hbox{for}\ \ \ m+1\leq\alpha\leq n.
 \\
 \end{split}
 \end{equation}
\end{corollary}

\begin{corollary} Let $x:M\rightarrow N$ be an immersed submanifold as above. Then $W(M)\geq0$. It is a Willmore submanifold with $W(M)\equiv0$ if and only if
it is a totally umbilic submanifold of $N$.\end{corollary}
\begin{proof} The inequality $W(M)\geq0$ is direct from the definition of $W(M)$. Then we have that $W(M)=0$ if and only if $\rho=0$ if and only if $x$ is totally umbilic.

 To show that totally umbilic submanifolds are Willmore\footnote{ One can also prove this in another way. Since totally umbilic submanifold attains the minimum of  Willmore functional, it is a stationary submanifold with respect to Willmore functional, i.e., Willmore submanifold.}, now we have that
 $\rho=0$ when $m>2$ and $h^{\alpha}_{11}=h^{\alpha}_{22},h^{\alpha}_{12}=0,$ for all $3\leq\alpha\leq n$ when $m=2$.
 So from \eqref{willmore}, $x$ is Willmore when $m>2$. When $m=2$, by use of $H^{\alpha}=h^{\alpha}_{11}=h^{\alpha}_{22},h^{\alpha}_{12}=0,$ we see that Willmore equation \eqref{willmore-s} holds.
\end{proof}

For more detailed description of totally umbilic submanifolds, we refer to the books by B.Y. Chen \cite{Chen0,Chen1}.

In \cite{Sou-Tou}, Souamand and Toubiana classified  totally unbilic surfaces in homogeneous 3-manifolds $\mathbb{M}^2(\kappa)\times\mathbb{R}$, $\mathbb{M}^3(\kappa,\tau)$(with $\kappa,\tau$ constants), and $Sol_3$.
Recall from \cite{Sou-Tou} that $\mathbb{M}^2(\kappa)$ is a complete, simply connected surface with constant curvature $\kappa
\in \mathbb{R}$. $\mathbb{M}^3(
\kappa,\tau)$ is a fibration over $\mathbb{M}^2(\kappa)$ with geodesic fibers. The unit vector field $\xi\in \Gamma(T\mathbb{M}^3(\kappa,\tau))$ tangent to the fibers is a Killing field and satisfies:
$$\bar{\nabla}_{X}\xi=\tau\cdot (X\times \xi), \ \forall X\in  \Gamma(T\mathbb{M}^3(\kappa,\tau)).$$
Here $\bar{\nabla}$ denotes the connection on $T\mathbb{M}^3(\kappa,\tau)$ and $\times$ denotes the cross product.
Note that $\xi$ defines the vertical direction of the Riemannian submersion $\mathbb{M}^3(\kappa,\tau)\rightarrow\mathbb{M}^2(\kappa)$.

When $\tau=0$, $\mathbb{M}^3(\kappa,\tau)$ is just the product $\mathbb{M}^2(\kappa)\times \mathbb{R}$. When $\kappa-4\tau^2=0$, one get the space forms. For the other cases with $\kappa(\kappa-4\tau^2)\neq0$, when $\kappa>0$, $\mathbb{M}^3(\kappa,\tau)$ is a Berger sphere; when $\kappa=0$, it is the Heisenberg space, $Nil_3$; when  $\kappa<0$, it is the universal covering of $PSL(2,\mathbb{R})$.

The $Sol_3$ geometry is a 3-dimension Lie group with a left-invariant metric, which is isometric to $(\mathbb{R}^3,\tilde{g})$ with $\tilde{g}=e^{2z}dx^2+e^{-2z}dy^2+dz^2$. The group product is given by
$$(\tilde{x},\tilde{y},\tilde{z})\star(x,y,z):=(e^{-\tilde{z}}x+\tilde{x},e^{\tilde{z}}y+\tilde{y},z+\tilde{z}).$$

For more details on homogeneous 3-manifolds we refer to \cite{Bonahon}, \cite{Sou-Tou}.

From Corollary 2.3, and Theorem 1, Theorem 9, Proposition 11, Theorem 16, and Theorem 19 in \cite{Sou-Tou}, we derive that
\begin{corollary}{\em (\cite{Sou-Tou})}

(i) For any immersed surfaces $M$ in the 3-manifolds $\mathbb{M}^3(\kappa,\tau)$, with $\tau(\kappa-4\tau^2)\neq0$, the Willmore functional
$W(M)>0$.

(ii) Let $M$ be a closed surface, and $\mathbb{M}^2(\kappa)$ be the 2-dimensional space form with constant curvature $\kappa$. Then for any immersion $x: M\rightarrow \mathbb{M}^2(\kappa)\times \mathbb{R}$, $W(M)\geq0$. And $W(M)=0$ if and only if $M=S^2$ and $x$ is totally umbilic.

(iii) Let $M$ be a closed surface. For any immersion $x: M\rightarrow Sol_3$, $W(M)>0$.
\end{corollary}

\section{Willmore functional of tori in Riemannian manifolds}

Now let us focus on the case that $M$ is a 2-dimensional torus. Then
$m=2$ and the Willmore functional of $x:M^2\rightarrow N^n$ is just
\begin{equation}W(x)=\int_M(|\vec{H}|^2-K+\tilde{K})dM=\int_M\sum_{\alpha}(\frac{1}{4}(h^{\alpha}_{11}-h^{\alpha}_{22})^2+
(h^{\alpha}_{12})^2)dM,\end{equation}
which also coincides with the
definition in \cite{Barros}. Here $K$ is the Gauss curvature of $M$ and $\tilde{K}=\tilde{K}(TM)$ is
the sectional curvature of $N$ as to the tangent plane $TM$.

For the minimum of $W(T^2)$ in a
Riemannian manifold, first we have that

\begin{theorem}Let $x:T^2\rightarrow N$ be an
immersed torus in a n-dimensional Riemannian manifold $N$
($n\geq3$).

(i). $W(T^2)\geq0.$ And $W(T^2)=0$ if and only if $x(T^2)$ is a totally umbilic torus in
$N $.

(ii). Assume that $N$ is of positive sectional curvature. Then
$N$ admits no totally umbilic closed surface with positive genus. As a special case, for any immersion
$x: T^2\rightarrow N $, $W(T^2)>0$.

(iii) Assume that $N$ is of non-negative sectional curvature. Then any totally umbilic torus $x:T^2\rightarrow
N $ (if exist) must be flat, totally geodesic in $N$.
\end{theorem}
\begin{proof}
(i) is direct.

For (ii), let $M^2$ be a compact Riemann surface with genus $g>0$. Suppose that $x:M^2\rightarrow$ is an immersed totally umbilic
surface. Then
$$ h^{\alpha}_{11}=h^{\alpha}_{22},\ h^{\alpha}_{12}=0,\ \alpha=3,\cdots,n.$$
The Gauss equation becomes
$$K=\tilde{K}+\sum_{\alpha}(h^{\alpha}_{11})^2.$$
Thus, $K>0$ on $M^2$, leading to $\int_{M^2}KdM>0$. This contradicts
with the Gauss-Bonnet formula $\int_{M^2}KdM=2\pi(2-2g)=0$.

For (iii),  $K\geq0$ and $\int_{M^2}KdM=0$ force $K\equiv0$. And since $\tilde{K}\geq0$, $\tilde{K}=\sum_{\alpha}(h^{\alpha}_{11})^2=0$, i.e., $II\equiv0$. That is, $x$ is a flat, totally geodesic immersion.
\end{proof}

\begin{example} Consider the special orthogonal group $(SO(n),g)$, $n\geq4$, with $g$ its natural induced metric, which is a Riemannian manifold with non-negative sectional curvature. So any 2-torus with $W(T^2)=0$ in
$SO(n)$ is totally geodesic and then is a 2-torus subgroup of some
maximal torus subgroup of $SO(n)$.
\end{example}

\begin{example}
In
\cite{Barros}, Barros gave torus $T^2\rightarrow (S^3,g_t)$ with $W(T^2)=2\pi^2t^2>0$, which is
also a minimal surface in $S^3$ with respect to the metric $g_t$. He called them Clifford torus \cite{Barros}. So when $t\rightarrow 0$, there exist torus in $(S^3,g_t)$ with $W(T^2)<2\pi^2$. Here we want to show that $W(T^2)>0$ for any torus in $(S^3,g_t)$.
To see this, we recall the definition of $(S^3,g_t)$. The Hopf
fibration $\pi: S^3\rightarrow S^2$ is also a Riemannian submersion
when $S^3$ and $S^2$ are equipped with the canonical metric
respectively. From the basic notions of Riemannian submersion in
 \cite{ONeill-book} (see also \cite{Barros}), at any point $p\in S^3$, we decompose $T_pS^3$ into vertical part $\mathcal{V}_p=T_p(\pi^{-1}(p))$ and horizontal part $\mathcal{H}_p=(T_p (\pi^{-1}(p)))^{\perp}$. So we have that
$$T_pS^3= \mathcal{V}_p\oplus\mathcal{H}_p,\ \pi_*(\mathcal{V}_p)\equiv0,\ \pi_*(\mathcal{H}_p)\cong T_{\pi(p)}S^2.$$
Now define $g_t$ as
\begin{equation}\left\{\begin{split}
&g_t(X,Y)|_{p}:=t^2g(X,Y)|_{p},\ \forall\ X,Y \in \mathcal{V}_p,\\
&g_t(Z,W)|_{p}:=g(Z,W)|_{p},\ \forall\ Z,W \in \mathcal{H}_p,\\
&g_t(X,Z)|_{p}:=0,\   \forall\ X \in \mathcal{V}_p,\ \forall\ Z \in \mathcal{H}_p.\\
\end{split}
\right.
\end{equation}
Note that $g_1=g$, and this deformation of metric keeps the horizontal part invariant, which just derives the Berger sphere.
For more
details, we refer to \cite{Barros}, \cite{Bonahon}, \cite{Sou-Tou}. From the definition of $g_t$, we
can see that $(S^3,g_t)$ has positive sectional curvature as
$(S^3,g)$. Then as a corollary, we derive that
\end{example}

\begin{corollary}Let $x:T^2\rightarrow (S^3,g_t)$ be an immersion
with $(S^3,g_t)$ defined as above. Then $W(T^2)>0$.
\end{corollary}

\begin{remark}
It is natural for us to ask the existence problem of torus minimizing
Willmore functional in $(S^3,g_t)$. When $t=1$, this reduces to the work of L.
Simon in \cite{Simon}. And we note that Kuwert etc. proved the existence problem of minimizing Willmore functional when the closed surface in $S^n$ is of genus $n>1$ in \cite{B-Ku}.

For the general cases, the existence problem will be a more complicated problem. For example, it is easy to show that even by doing a small change of a Riemannian manifold, one may obtain a metic with minimizing torus.
\end{remark}

\begin{proposition}Let $(N,g)$ be a 3-dimensional complete smooth Riemannian manifold. There exists a  smooth, complete
metric $\tilde{g}$ such that there exists an immersed torus
$x:T^2\rightarrow (N,\tilde{g})$ with $W(T^2)=0$. Here
$\tilde{g}=g$ on $N\backslash U$ with $U\subset N$ an $\varepsilon$-geodesic ball.
\end{proposition}

\begin{proof} Obviously we just need to prove this for $n=3$.
Suppose $U\subset N$ is an $\varepsilon$-geodesic ball. Then let $f:D^3\rightarrow U$ be a diffeomorphism. Consider a
subset $T^2\times[0,1]\subset D^3$ such that every $T^2\times t$ is
embedded in $D^3$ for any $t\in[0,1]$. Let $a(t)$ be a smooth
function satisfying
\begin{equation*}
a(t)=\left\{ \begin{aligned}\ 0, &\ \ 0\leq t\leq\frac{1}{5},\ \frac{4}{5}\leq t\leq 1,\\
\ 1, & \ \ \frac{2}{5}\leq t\leq\frac{3}{5}.\\
 \end{aligned} \right.
                          \end{equation*}
 Now we define on $f(T^2\times(0,1))\subset U$ the product metric $g_1=g_{T^2}\times g_{t}$
 with $g_{T^2}$ the canonical flat metric on $T^2$ and $g_{t}$ the
 natural metric on $(0,1)$.

We define the new metric $\tilde{g}$ on $N$ as follows:
\begin{equation}
\tilde{g}=\left\{ \begin{aligned}\ \ \ \ \ \ & g, \ \ \ \ \ \ \ \ \ \hbox{on}\ N\backslash f(T^2\times(0,1)),\\
\ &(1-a(t))g+a(t)g_1,  \ \hbox{on}\ f(T^2\times(0,1)).\\
 \end{aligned} \right.
                          \end{equation}
\end{proof}

\begin{remark}Notice that the truncation function $a(t)$ used in our proof is smooth and not real analytic
and that Willmore surfaces in space forms are in general real
analytic because of the ellipticity of Willmore equation, we have
the question as follows:\\

 {\bf Question:} Let $M^2$ be an oriented closed
surface. Does there exist a complete real analytic Riemannian metric
$g$ on $R^n$ such that there exists an embedding $x:M^2\rightarrow
(R^{n},g)$ with $W(M^2)=0$ if $M^2$ is not diffeomorphic to $S^2$ ?
\\\end{remark}

For the case that $M^2=T^2$ and $n=4$, the answer is positive.

\begin{proposition}Let $(R^4,g_{\lambda,\mu})$ be a complete Riemannian manifold with
$$g_{\lambda,\mu}=\frac{e^{\lambda(x_1^2+x_2^2)}(dx_1^2+dx_2^2)}{(1+x_1^2+x_2^2)^2}+\frac{e^{\mu(x_3^2+x_4^2)}(dx_3^2+dx_4^2)}{(1+x_3^2+x_4^2)^2}
, 0<\lambda,\mu\leq3-2\sqrt{2}.$$   Then the torus
$$x:T^2\rightarrow
(R^{4},g_{\lambda,\mu}),\ x(u,v)=(r_{\lambda}\cos u,r_{\lambda}\sin u,r_{\mu}\cos v,r_{\mu}\sin v),$$
with
$$r_{\lambda}=\sqrt{\frac{1-\lambda\pm\sqrt{\lambda^2-6\lambda+1}}{2\lambda}},\ r_{\mu}=\sqrt{\frac{1-\mu\pm\sqrt{\mu^2-6\mu+1}}{2\mu}},$$
 is a totally geodesic torus,
i.e. Willmore torus with $W(T^2)=0$.
\end{proposition}

\section{Torus in product manifolds $M^2\times S^1$}

In this section, we consider the minimum of $W(T^2)$ in some product
Riemannian manifolds $M^2\times S^1$.

First we assume that $M=S^2$ (of constant positive curvature $1$).

\begin{proposition}The Willmore funcitional $W(x)=\int_{T^2}(H^2-K+\tilde{K})dM$ of an immersed torus $x:T^2\rightarrow S^{2}\times
S^{1}$ is non-negative. And $W(x)=0$ if and only if $x(T^2)$ is a
totally geodesic torus in $S^2\times S^1$. So it is congruent to the
Clifford torus $x(u,v)=(\cos u,\sin u,0,\cos v,\sin v)$ under
isometric transforms of $S^2\times S^1$.
\end{proposition}
\begin{proof}
 Since $S^2\times S^1$ is of non-negative curvature, $x$ is a flat totally geodesic torus in
 $S^2\times S^1$. We also have that $\tilde{K}=0$, showing that it is a product surface.
\end{proof}

 From the proof above, noticing the key point is that $S^2$ has
closed geodesic and positive sectional curvature, we have

\begin{corollary}Let $M$ be a compact Riemannian manifold with positive sectional
curvature with at least one closed geodesic $\gamma$. Then the
torus $\gamma\times S^1\subset M\times S^{1}$ has $W(T^2)=0$.
And any $T^2$ in $M\times S^{1}$ with $W(T^2)=0$ must be of such
form.
\end{corollary}

\begin{remark} When $M=T^2$, there are infinite non-congruent geodesics
in $T^2$, giving totally geodesic tori with different conformal
structures.
\end{remark}

Now let us suppose that $M=\mathbb{R}^2$. Different from $S^2$ cases, we have

\begin{theorem} Let $(\mathbb{R}^2\times S^1,g_{R^2}\times g_{S^1})$ be the product of Euclidean surface and $S^1$.

(i). For any  2-torus
 in $\mathbb{R}^2\times S^1$,  $W(T^2)>0$.

(ii). The Willmore functional of the torus $x_t=(t\cos u,t\sin u, \cos v,\sin
v)$ in $\mathbb{R}^2\times S^1$ is $W(x_t)=\frac{\pi^{2}}{t}$, which is tending to $0$ when $t\rightarrow\infty$.\\
So there exists no torus having the least
Willmore functional $W(T^2)$.
\end{theorem}

\begin{proof} For (i), first we see that $x$ is a flat totally geodesic torus in
 $M\times R^1$. While there exists no flat totally geodesic torus in  $R^2\times S^1$, since flat totally geodesic surface of $R^2\times S^1$ must be of the form
 $\gamma\times S^1$ with $\gamma$ being a line in $R^2$, or of $R^2\times p$ for some $p\subset S^1.$

For (ii), to calculate  $W(x_t)$ of $x_t$, let $$x:T^2\rightarrow R^2\times
S^{1}\subset R^2\times R^2.$$
 We have $$x_u=(-t\sin u,t\cos u,0,0),\
x_v=(0,0,\cos v,\sin v).$$ Set $$\tilde{x}= (t\cos u,t\sin u,-\cos
v,-\sin v),\ n=( \cos u,\sin u,0,0).$$
Then
$$h_{11}=\frac{1}{t^2}\langle x_{uu},n\rangle=\frac{1}{t},\
h_{12}=\frac{1}{t}\langle x_{uv},n\rangle=0,\ h_{22}=\langle
x_{vv},n\rangle=0.$$ So
$$W(x_t)=\int_{T^2}\frac{1}{4t^2}dM=\int_{0}^{2\pi}\int_{0}^{2\pi}\frac{1}{4t^2}tdudv=\frac{\pi^2}{t}.$$
\end{proof}

From the proof, it is direct to derive
\begin{corollary} Let $M^{n}$ be an n-dimensional Riemannian manifold with positive sectional curvature. Assume that there exists a  closed
geodesic $\gamma$ in $M^n$.
Then the Willmore funcitional $W(x)$ of an immersed torus $x$ in the
product Riemannian manifold $R^2\times M^{n}$ is positive. And the
Willmore functional of the tori $x_t=(t\cos u,t\sin u,
\gamma)\subset R^2\times M^{n}$ $W(x_t)=\frac{\pi^{2}}{t}$. When
$t\rightarrow\infty$, $W(x_t)\rightarrow0$. So there exists no torus
having the least Willmore functional $W(T^2)$.
\end{corollary}

 \begin{remark}

(i). We note that the non existence of tori attaining the
minimum of $W(T^2)$ depends on both the manifolds and
metrics. If we set a metric $\tilde{g}$ on $R^2$ such that
$(R^2,\tilde{g})$ has closed geodesics, then $(R^2\times
S^{1},\tilde{g}\times g_{S^1})$ will have a totally geodesic torus.

(ii). One can see that $\mathbb{R}^2\times S^1$ can not be conformally compactificated into any compact Riemannian manifold. This may be a reason
of the non-existence of minimizing torus in $\mathbb{R}^2\times S^1$. So it is natural to {\em conjecture} that

{\em  For a compact Riemannian manifold $N$, let $M$ be a closed surface. Then there exists an immersion $x:M\rightarrow N$ such that $x$ minimize the Willmore functional among all immersions from $M$ to $N$.}

(iii). Since the compact property is a topological property, one may also be interested in some curvature conditions. While noticing that Willmore functional is a conformal invariant, it is natural to consider the Weyl curvature restriction. Since space forms and $\mathbb{R}^2\times S^1$ are all conformal flat, it seems that the Weyl curvature is not a enough choice.

(iv). Notice that when $t\rightarrow \infty$, the conformal structure of $T^2$ is changing too. If we fix the conformal structure of $T^2$, the proof above does not work. For example, if we assume that $T^2$ is conformal to $S^{1}(1)\times S^1(a)$, with $a=\frac{p}{q}$ some positive rational number, then to get the right torus $t$ must be $kq$ for some $k\in \mathbb{N}$. Then the Willmore functional should be
$$W(x_t)=\int_{T^2}\frac{1}{4t^2}dM=\int_{0}^{2\pi }\frac{1}{4t^2}tdu\int_{0}^{2\pi ta}dv=a\pi^2,$$
which does not depend on the choice of $t$. So we also can conjecture that for all compact Riemann surfaces conformally immersed into a complete Riemannian manifold, there exists an immersion minimizing the Willmore functional.
\end{remark}

Another interesting case is that $M$ is of constant negative sectional curvature.
As to this case, we obtain
\begin{proposition} The Willmore funcitional $W(x)$ of an immersed torus $x:T^2\rightarrow H^2(-c)\times
S^{1}$ is positive, where $H^2(-c)$ is the space form with constant
curvature $-c<0$. Let $\gamma(s)\subset H^2(-c)$ be a immersed
closed curve with arc length parameter $s$. Consider the torus
$x=(\gamma, \cos t,\sin t)$, we have that
$W(x)=\frac{\pi}{2}\int_{\gamma}k^2ds\geq2\pi^2\sqrt{c}$.
The equality holds if and only if $\gamma$ is congruent to a geodesic
circle of radius $\frac{\sinh^{-1}1}{\sqrt{c}}$ and at this case $x$ is a Willmore torus in  $H^2(-c)\times
S^{1}$.
\end{proposition}
\begin{proof} The first result is the same as above. We just need to compute $W(x)$. Suppose that
$$\gamma_s=\alpha,\ \alpha_s=k\beta+\frac{1}{c^2}\gamma,\ \beta_s=-k\alpha.$$
Here $k$ is just the curvature of $\gamma$. So we can obtain that
for $x$, $h_{11}=k, \ h_{12}=h_{22}=0$, leading to
$$W(x_t)=\int_{T^2}\frac{k^2}{4}dM=\int_{\gamma}\int^{2\pi}_{0}\frac{k^2}{4}dtds=\frac{\pi}{2}\int_{\gamma}k^2ds.$$

By Theorem 4.1 in \cite{Langer-Singer1984}, for any closed curve in
$H^{2}(-c)$, $\int_{\gamma}k^2ds\geq4\pi\sqrt{c}$ and equality holds
if and only if $\gamma$ is congruent with a geodesic circle of radius
$\frac{\sinh^{-1}1}{\sqrt{c}}$. Similar to the discussion in
\cite{Barros}, in this case, $x=(\gamma, \cos t,\sin t)$ is a
Willmore torus in $H^2(-c)\times S^{1}$ due to the principle of
symmetric criticality in \cite{Pa} (see also \cite{Barros99,Barros88}).
\end{proof}

{\bf Acknowledgments}
The author is thankful to Professor Changping Wang, Zejun, Hu, and Xiang Ma for valuable discussion and suggestions.
The author is also thankful to the referee for many valuable suggestions which make the paper much more readable.

\def\refname{Reference}

 \vspace{5mm}  \noindent Peng
Wang, {\small\it Department of Mathematics, Tongji University,
200092 Shanghai, People's Republic of China.} e-mail: {\sf
netwangpeng@tongji.edu.cn}
\end{document}